\documentclass[11pt]{amsart}
\usepackage{amssymb}
\usepackage{dsfont}

\newtheorem{theorem}{Theorem}[section]
\newtheorem{lemma}[theorem]{Lemma}

\newtheorem{corollary}[theorem]{Corollary}

\theoremstyle{definition}
\newtheorem{df}{Definition}

\newtheorem{remark}[df]{Remark}

\newcommand{\I}{\mathcal I}

\newcommand{\N}{\mathbb N}

\newcommand{\card}{\operatorname{card}}

\author{Micha{\l} Pop\l awski}
\address{Institute of Mathematics, University of Silesia,
Bankowa 14, 40-007 Katowice, Poland}
\email {michal.poplawski@us.edu.pl}
\title[Subseries and rearrangements]{Ideal boundedness of subseries and rearrangements in Banach spaces \\ vs Banach spaces possessing a copy of $c_0$}
\date{}
\date{}
\begin{document}
\begin{abstract}
Suppose that $X$ is a Banach space. We will show that $X$ does not contain a copy of $c_0$ if and only if for each series which is not unconditionally convergent in $X$ respective sets coding all bounded subseries and rearrangements are meager. We use Bessaga-Pe{\l}czy{\'n}ski $c_0-$Theorem and concept of uniformly unconditionally bounded series. Moreover we prove similar result for the ideal boundedness for a class of Baire ideals using Talagrand's characterisation.
\end{abstract}

\maketitle
\section{Introduction}

Denote by $Fin$ the family of all finite subsets of $\N$. Recall that an ideal $\I$ on $\N$ is the nonempty subfamily of the family $P(\N)$ such that
\begin{enumerate}
\item $\emptyset \in \I,$
\item $A,B \in \I \Rightarrow A \cup B \in \I,$
\item $A \subset B \wedge B \in \I \Rightarrow A \in \I,$
\item $\N \notin \I,$
\item $Fin \subset \I,$
\end{enumerate} 
where the last two conditions are not always included in the definition of ideal, but they are often considered as useful properties.  

Since $\I \subset P(\N)$ and $P(\N)=\{0,1\}^{\N}$ has natural topological structure as product space of discrete spaces $\{0,1\}$,  then we may investigate topological properties of ideals. Especially, recall that a subset $A$ of a metric space $X$ has the Baire property iff $A$ is symmetric difference of a meager set $N$ and an open set $U$. Family of these sets constitutes a $\sigma-$algebra on $X$.
The following result (due to Jalali-Naini and Talagrand; see \cite{JN}, \cite{Ta}) gives a useful characterization of
ideals with the Baire property.
\begin{lemma}\label{Tal}
An ideal $\I$ on $\N$ has the Baire property if and only if
there is an infinite sequence $n_1<n_2<\dots$ in $\N$ such that no member of $\I$ contains infinitely many intervals $[n_i,n_{i+1})\cap\N$.
\end{lemma} 
An important example of an ideal distinct than $Fin$ and having the Baire property is the density ideal $\I_d:=\left\{A \subset \N \colon \frac{\card(A \cap \{1,\ldots,n\})}{n} \to 0\right\}.$

We say that a sequence $(x_n)_n$ in a normed space is $\I-$bounded if there is $M>0$ with $\{n \in \N \colon \|x_n\|>M\} \in \I$. Evidently if we set $\I=Fin$ then we get usual notion of boundedness of a sequence. Moreover it can be shown that $\I-$convergent sequence is $\I-$bounded, too (\cite{KMSS}, Theorem 2.2, by modification).  

Put $S:=\left\{s \in \N^{\N} \colon s \textrm{ is increasing }\right\}, P:=\left\{p \in \N^{\N} \colon p \textrm{ is a bijection }\right\}.$ It is easy to see that both sets are polish spaces, alike $\{0,1\}^{\N}.$ 
If $\sum x_n$ is a series in a normed space, than we may think that subseries of $\sum x_n$ is generated by a sequence $t \in \{0,1\}^{\N}$ or by a sequence $s \in S$ in such fashions: $\sum t(n)x_n, \sum x_{s(n)}$, likewise $\sum x_{p(n)}$ is rearrangement if $p \in P.$

In this paper we are interested in ideal boundedness, hence we will only mention that following theorem has its equivalents in a case of ideal convergence in \cite{BPW, BPW2}.
\begin{theorem}(\cite{BPW}, Theorem 4.1, Corollary 4.2) 
Suppose that $\sum x_n$ is a series which is not unconditionally convergent in a finite-dimensional Banach space $X$. Let $\I$ be an ideal with the Baire property. Then the sets
$$E(\I,(x_n))=\left\{s \in S \colon \left(\sum_{i=1}^n x_{s(i)}\right)_n \textrm{ is } \I-\textrm{bounded} \right\},$$
$$F(\I,(x_n))=\left\{p \in P \colon \left(\sum_{i=1}^n x_{p(i)}\right)_n \textrm{ is } \I-\textrm{bounded} \right\},$$
are meager in $S$ and $P$, respectively.
\end{theorem}
It is easy to see (analyzing the proof) that the above theorem holds changing coding of subseries from $S$ to $\{0,1\}^{\N}.$

In (\cite{BPW}, Example 3) the authors constructed a series $\sum x_n$ in $c_0$ which is not unconditionally convergent, but $E(Fin,(x_n))=S, F(Fin, (x_n))=P$, then both sets are evidently comeager. (Note that this example works for any Banach space possessing copy of $c_0$ and for any ideal $\I$, since boundedness implies ideal boundedness.) The above facts lead us to the question (\cite{BPW}, Question 4.3): for which (infinite-dimensional) Banach spaces for each series which is not unconditionally convergent both sets $E(\I,(x_n)), F(\I,(x_n))$ are meager (in a case of Baire ideal $\I$)?

\section{Main results}

We will say that a series $\sum x_n$ in a Banach space is uniformly unconditionally bounded if there is $M>0$ such that for each $p \in P$ and $n \in \N$ we have $\left\| \sum_{i=1}^n x_{p(i)}\right\| \leq M.$ 
The following result should, in all likelihood, be known.

\begin{lemma}\label{wr}
Suppose $\sum x_n$ is a series in a Banach space. The following conditions are equivalent:
\begin{enumerate}
\item $E(Fin,(x_n))=S,$
\item there is $M'>0$ s.t. for all $s \in S$ and $n \in \N$ we have $\left\|\sum_{i=1}^{n} x_{s(i)} \right\| \leq M',$
\item $\sum x_n$ is uniformly unconditionally bounded,
\item $F(Fin,(x_n))=P.$
\end{enumerate}
\end{lemma}
\begin{proof}
"(1) $\Rightarrow$ (2) "

Suppose (1) holds but there is no such $M'$ like in (2). Take an increasing sequence $(s_1,\ldots,s_{k_1})$ such that $\|\sum_{i=1}^{k_1} x_{s_i}\|>0$. By the assumption, a series $\sum_{j=s_{k_1}+1}^{\infty} x_j$ does not satisfy (2), too. Then we can find indices $j_1<j_2<\ldots<j_{k_2}$ with $j_1>s_{k_1}$ such that $\|\sum_{i=1}^{k_2} x_{j_i}\|  \geq 3 \|\sum_{i=1}^{k_1} x_{s_i}\|.$ Since $\|x+y\| \geq |\|x\|-\|y\||,$
then if we put $s_{k_1+1}=j_1,s_{k_1+2}=j_1+1,\ldots,s_{j_{k_2}-k_1+1}=j_{k_2}$ we get $\|\sum_{i=1}^{j_{k_2}-k_1+1} x_{s_i}\| \geq 2\|\sum_{i=1}^{k_1} x_{s_i}\|.$

Proceeding this way we can define $s \in S$ such that there is $u \in S$ with $\|\sum_{i=1}^{u(n)} x_{s(i)}\| \geq 2^n\|\sum_{i=1}^{k_1} x_{s_i}\|$ for each $n \in \N$ which shows that $s \notin E(Fin,(x_n)).$

"(2) $\Rightarrow$ (3) "

Suppose that all subseries have common bound $M'>0$ and fix $p \in P$. Take a sequence $(s_k)_k$ of elements of $S$ with property $\{s_k(1),\ldots,s_k(k)\}=\{p(1),\ldots,p(k)\}$ for each $k \in \N$. Then for all $k \in \N$ we obtain $\|\sum_{i=1}^{k} x_{p(i)}\|=\|\sum_{i=1}^{k} x_{s_k(i)}\| \leq M'.$

"(3) $\Rightarrow$ (4) "
Obvious.

"(4) $\Rightarrow$ (1) "
Suppose (1) does not hold and fix $s \in S$ with unbounded $\sum x_{s(n)}.$ We will define an unbounded rearrangement of $\sum x_n.$ Take $n_1 \in \N$ with $\left\|\sum_{i=1}^{n_1} x_{s(i)}\right\| \geq 1.$ Find $k_1>n_1$ and bijection $p_1 \colon \{1,\ldots,k_1\} \to \{1,\ldots,k_1\}$ extending a sequence $(s(1),\ldots,s(n_1)).$ By the assumption a series $\sum_{i=k_1+1}^{\infty} x_{s(i)}$ is unbounded, too. Then we can pick $n_{2}>k_{1}$ with $\left\|\sum_{i=1}^{k_1} x_{p_{1}(i)}+\sum_{i=k_{1}+1}^{n_2} x_{s(i)}\right\| \geq 2$. Find $k_{2}>n_{2}$ and a bijection $p_2 \colon \{1,\ldots,k_2\} \to \{1,\ldots,k_2\}$ extending $p_1$ and $(s(1),\ldots,s(n_2))$ simultaneusly. Proceeding this concept \textit{ad infinitum} we infer that $p:=\bigcup_{n \in \N} p_n \in P$ satisfy inequality $\left\|\sum_{i=1}^{n_j} x_{p(i)}\right\| \geq j$ for each $j \in \N,$ thereby $p \notin F(Fin,(x_n)).$

\end{proof}

\begin{theorem} \label{dch}
Suppose $\sum x_n$ is a series in a Banach space. If $E(Fin,(x_n)) \neq S$ (equivalently if $F(Fin,(x_n)) \neq P$), then $E(Fin,(x_n)),F(Fin,(x_n))$ are meager in $S$ and $P$, respectively.
\end{theorem}
\begin{proof}
Suppose there is $s' \in S \setminus E(Fin,(x_n)).$ Write $E(Fin,(x_n))=\bigcup_{m \in \N} A_m,$ where $A_m=\{s \in S \colon \left\|\sum_{i=1}^n x_{s(i)}\right\| \leq m \mbox{ for all } n \in \N\}$ for $m \in \N$. Fix $m \in \N$. 

We will show that $A_m$ is closed. Indeed, take a sequence $(s_n)_n \in A_m$ convergent to some $s \in S$. Fix $n \in \N.$ Take $k_0 \in \N$ such that for all $k \geq k_0$ and $i \leq n$ $s_k(i)=s(i)$. Then $\left\|\sum_{i=1}^n x_{s(i)}\right\|=\left\|\sum_{i=1}^n x_{s_k(i)}\right\| \leq m$ since $s_k \in A_m.$ Consequently, $s \in A_m.$

Now, we will show that $A_m$ has an empty interior, which suffices to infer that $A_m$ is nowhere dense. Consider basic open set 
$$U:=\{s \in S \colon s \mbox{ extends } (s_1,\ldots,s_k)\}$$
for some increasing sequence $(s_1,\ldots,s_k)$ of positive integers. Find the smallest $l \in \N$ with $s'(l)>s_k$ and define $u:=(s_1,\ldots,s_k,s'(l),s'(l+1),s'(l+2),\ldots).$ Then $u \in U \setminus A_m,$ which proves that no open set is contained in $A_m.$

For the second part, suppose there is $p' \in P \setminus F(Fin,(x_n))$ and write $F(Fin,(x_n))=\bigcup_{m \in \N} D_m,$ where $D_m=\{p \in P \colon \left\|\sum_{i=1}^n x_{p(i)}\right\| \leq m \mbox{ for all } n \in \N\}$ for $m \in \N$. Similar calculations as above show that all $D_m$'s are closed. Fix $m \in \N$.

We will check that $D_m$ is nowhere dense. Consider basic open set
$$V:=\{p \in P \colon p \mbox{ extends } (p_1,\ldots,p_k)\}$$
for some injective sequence $(p_1,\ldots,p_k)$ of positive integers. Pick the smallest $l_1 \in \N$ such that $p'[\{1,\ldots,l_1-1\}] \supset \{p_1,\ldots,p_k\}$ and $j_1>l_1$ with $\|x_{p_1}+\ldots+x_{p_k}+x_{p'(l_1)}+\ldots+x_{p'(j_1)}\|>m$. Next, extend a sequence $(p_1,\ldots,p_k,p'(l_1),\ldots,p'(j_1))$ to bijection $p_1 \colon \{1,\ldots,k_1\} \to \{1,\ldots,k_1\}$ for some $k_1 \in \N$. Now, pick the smallest $l_2 \in \N$ such that $p'[\{1,\ldots,l_2-1\}] \supset \{1,\ldots,k_1\}$ and $j_2>l_2$ with $\|x_{{p_1}_1}+\ldots+x_{{p_1}_{k_1}}+x_{p'(l_2)} + \ldots+x_{p'(j_2)}\|>m$. Again, extend a sequence $({p_1}_{1},\ldots,{p_1}_{k_1},p'(l_2),\ldots,p'(j_2))$ to a bijection $p_2 \colon \{1,\ldots,k_2\} \to \{1,\ldots,k_2\}$ for some $k_2 \in \N.$ This inductive procedure produces $p \in V \setminus D_m.$
\end{proof}

Now, we will move to the case of ideal boundedness.

\begin{theorem}\label{dych}
Suppose $\sum x_n$ is a series in a Banach space with $\liminf \|x_n\|=0$ and $\I$ is an ideal with the Baire property. If $E(Fin,(x_n)) \neq S$ (equivalently if $F(Fin,(x_n)) \neq P$), then $E(\I,(x_n)),F(\I,(x_n))$ are meager in $S$ and $P$, respectively.
\end{theorem}

\begin{proof}
Let $(n_k)_{k \in \N}$ be a sequence from Talagrand's characterisation and $I_k:=[n_k,\ldots,n_{k+1}-1) \cap \N$.
Take $u \in S$ such that $\sum x_{u(n)}$ is unbounded. 

For each $m \in \N$ define 
$$B_m:=\left\{s \in S \colon \exists_{k>m} \forall_{l \in I_k} \left\| \sum_{i=1}^l x_{s(i)}\right\|>m \right\}.$$
Note that $\bigcap_{m \in \N} B_m \subset S \setminus E(\I,(x_n)).$ Indeed, take any $M>0$ and take any $m_1 \in \N$ with $m_1>M$. Since $s \in B_{m_1}$ then there is $k(m_1)>m_1$ with  $\left\| \sum_{i=1}^l x_{s(i)}\right\|>m_1$ for all $l \in I_{k(m_1)}.$ Take any integer $m_2>n_{k(m_1)+1}-1$ and note that since $s \in B_{m_2}$ then we have $\left\| \sum_{i=1}^l x_{s(i)}\right\|>m_2>m_1$ for all $l \in I_{k(m_2)}.$
Such procedure generates an increasing sequence $(m_j)_{j \in \N}$ of positive integers with $\left\| \sum_{i=1}^l x_{s(i)}\right\|>m_j>m_1$ for all $l \in I_{k(m_j)}$ and $j \in \N$ which shows that $\left\{n \in \N \colon \left\| \sum_{i=1}^l x_{s(i)}\right\|>m_1>M \right\}$ contains infinitely many intervals of the form $I_k.$ Therefore since $M>0$ is arbitrary we infer that (thanks to Lemma \ref{Tal}) $s \notin E(\I,(x_n)).$

Below we will show that each $B_m$ contains an open dense subset of $S$. Thereby, since $S$ is compleletely metrisable we get that $\bigcup_{m \in \N} B_m$ is comeager, and finally $E(\I,(x_n))$ is meager.

Fix $m \in \N$ and basic open set $U$ in $S$:
$$U=\left\{w \in S \colon w \textrm{ extends } (s_1,\ldots,s_r) \right\}$$
for some increasing sequence $(s_1,\ldots,s_r)$ of positive integers, where $r>m.$ Since $\sum_{n=1}^{\infty} x_{u(n)}$ is unbounded then $\sum_{i=r+1}^{\infty} x_{u(i)}$ is unbounded, too. Hence, we can pick $l_r>r$ with $$\left\|\sum_{i=1}^r x_{s_i}+\sum_{i=r+1}^{l_r} x_{u(i)}\right\|>m+1.$$ Find the smallest $k \in \N$ with $n_k>l_r$ and $\sum_{i=l_r+1}^{n_{k+1}-1} \|x_{v_i}\|<1$ for some increasing $(v_{l_r+1},\ldots,v_{n_{k+1}-1})$ such that $v_{l_r+1}>u(l_r)$ (it is possible since $\liminf \|x_n\|=0$). Then 
$$\left\|\sum_{i=1}^r x_{s_i}+\sum_{i=r+1}^{l_r} x_{u(i)}+\sum_{i=l_r+1}^{j} x_{v_i}\right\| \geq \left\|\sum_{i=1}^r x_{s_i}+\sum_{i=r+1}^{l_r} x_{u(i)}\right\|-\left\|\sum_{i=l_r+1}^{j} x_{v_i}\right\|>m+1-1=m$$ for each $j \in I_k.$
Put $$V_U:=\left\{w \in S \colon w \textrm{ extends } (s_1,\ldots,s_r,u(r+1),\ldots,u(l_r),v_{l_r+1},\ldots,v_{n_{k+1}-1})\right\} \subset U \cap B_m$$ and note that $\bigcup_{U \in \mathcal{B}}V_U$ is a dense (in $S$) open subset of $B_m$, where $\mathcal{B}$ is a base of topology in $S$ consisting of sets of the form $\left\{w \in S \colon w \textrm{ extends } (z_1,\ldots,z_r) \right\}$ for some finite increasing sequence $(z_1,\ldots,z_r).$

Now we will discuss the case of rearrangements. One can show that $\bigcap_{m \in \N} C_m \subset P \setminus F(\I,(x_n))$, where 
$$C_m:=\left\{p \in P \colon \exists_{k>m} \forall_{l \in I_k} \left\| \sum_{i=1}^l x_{p(i)}\right\|>m \right\}$$
for each $m \in \N.$ Now, we fix $m \in \N$ and we will check that $C_m$ contains an open basic set. Consider
$$U=\left\{w \in P \colon w \textrm{ extends } (p_1,\ldots,p_r) \right\}$$
for some $r>m.$ Since $E(Fin,(x_n)) \neq S$, then there is an unbounded rearrangement (see Lemma \ref{wr}) $\sum_{i=1}^{\infty} x_{t(i)}$ of $\sum_{i=1}^{\infty} x_i$ ($t \in P$). Put $z:=\max\{r,\max\{p_1,\ldots,p_r\}\}$ and note that $\sum_{i=z+1}^{\infty} x_{t(i)}$ is unbounded, too. Find $m_r>z$ with $\|\sum_{i=1}^r x_{p_i}+\sum_{i=z+1}^{m_r} x_{t(i)}\|>m+1.$ Pick the smallest $k \in \N$ with $n_k>r+m_r-z$ and find an increasing sequence $(v_{m_r+1},\ldots,v_{n_{k+1}-1-r+z})$ with $v_{m_r+1}>\max\{z,m_r,\max\{t(z+1),\ldots,t(m_r)\}\}$ and $\sum_{i=m_r+1}^{n_{k+1}-1-r+z}\|x_{v_i}\|<1.$ Then 
$$\left\|\sum_{i=1}^r x_{p_i}+\sum_{i=z+1}^{m_r} x_{t(i)}+\sum_{i=m_r+1}^{j-r+z} x_{v_i}\right\|>m$$
for each $j \in I_k.$ Then 
$$V_U:=\left\{w \in P \colon w \textrm{ extends } (p_1,\ldots,p_r,t(z+1),\ldots,t(m_r),v_{m_r+1},\ldots,v_{n_{k+1}-1-r-z})\right\} \subset U \cap C_m.$$ The conclusion is similar as the final reasoning for case of subseries.
\end{proof}

\begin{remark}
In Theorem \ref{dych} we can change the assumption $\liminf \|x_n\|=0$ to $\limsup \|x_n\|=\infty$, with the same thesis. It follows from the fact that in this case one can construct subseries $\sum_{n=1}^{\infty} x_{u(n)}$ of a series $\sum x_n$ that both sequences $(\|x_{u(n)}\|)_n,(\|\sum_{i=1}^{n} x_{u(i)}\|)_n$ tends increasingly to $\infty.$
\end{remark}

\begin{remark} \label{bw}
In a case of subseries determined by $0-1$ sequences we can omit condition $\liminf \|x_n\|=0.$
\end{remark}

Recall that Bessaga-Pe{\l}czy{\'n}ski $c_0-$Theorem (\cite{Kad},Theorem 6.4.1, p.85) states that a Banach space $X$ contains no isomorphic copy of $c_0$ iff every weakly absolutely convergent series is unconditionally convergent [iff every weakly absolutely convergent series is convergent].

\begin{lemma} \label{wac} (\cite{JD},Theorem 6, p.44 ($1. \Leftrightarrow 4.$))
A series $\sum x_n$ in a Banach space $X$ is weakly absolutely convergenct iff there is $C>0$ such that for each $n \in \N$ and finite sequence $t \colon \{1,\ldots, n\} \to \{-1,0,1\}$ we have $\|\sum_{i=1}^n t(i)x_i\| \leq C$.
\end{lemma}

\begin{theorem} 
Suppose that $\sum x_n$ is not unconditionally convergent series in a Banach space containing no copy of $c_0$. Then $E(Fin,(x_n)), F(Fin,(x_n))$ are meager.
\end{theorem}
\begin{proof}
Suppose $E(Fin,(x_n))$ or $F(Fin,(x_n))$ is nonmeager. Then by Theorem \ref{dch} $E(Fin,(x_n))=S, F(Fin,(x_n))=P$ and by Lemma \ref{wr} all subseries and rearrangements are bounded by a common constant, say $M>0$. Note that for each sequence $t \colon \{1,\ldots, n\} \to \{-1,0,1\}$ we can write $\|\sum_{i=1}^n t(i)x_i\|=\|\sum_{i=1}^n \chi_{A}(i)x_i+\sum_{i=1}^n \chi_{B}(i)t(i)x_i\| \leq \|\sum_{i=1}^n \chi_{A}(i)x_i\| + \|\sum_{i=1}^n -\chi_{B}(i)t(i)x_i\| \leq 2M$, if we put $A:=t^{-1}[\{1\}], B:=\{1,\ldots,n\} \setminus A$, because $\chi_{A}, -\chi_{B}t \colon \{1,\ldots,n\} \to \{0,1\}.$ It follows from Lemma \ref{wac} that $\sum x_n$ is weakly absolutely convergent and by Bessaga-Pe{\l}czy{\'n}ski $c_0-$Theorem $\sum x_n$ is unconditionally convergent.
\end{proof}

Using an above fact and Example 3. in \cite{BPW} we get following result. 
\begin{corollary} \label{wn}
A Banach space contains no copy of $c_0$ if and only if for each series which is not unconditionally convergent both sets $E(Fin,(x_n)), F(Fin,(x_n))$ are meager in $S,P$, respectively.
\end{corollary}
Note that the fact that $E(Fin,(x_n))$ is meager or a series $\sum x_n$ is weakly unconditionally convergent was proven (\cite{Lef},Corollary 2.2) using theory of summability methods of a series.

\begin{theorem}
Suppose $\I$ is an ideal with the Baire property. A Banach space $X$ contains no copy of $c_0$ if and only if for each series $\sum x_n$ in $X$ with $\liminf \|x_n\|=0$ which is not unconditionally convergent both sets $E(\I,(x_n)),F(\I,(x_n))$ are meager in $S,P$, respectively.
\end{theorem}
\begin{proof}
If Banach space $X$ contains a copy of $c_0$ then we use a construction from Example 3. in \cite{BPW}. Now, suppose that $X$ contains no copy of $c_0$ and $E(\I,(x_n))$ or $F(\I,(x_n))$ is nonmeager. Suppose that $E(Fin,(x_n)) \neq S$. Then by Lemma \ref{wr} and Theorem \ref{dych} both sets $E(\I,(x_n)), F(\I,(x_n))$ are meager, contradiction. Then $E(Fin,(x_n))=S$. By Lemma \ref{wr} $F(Fin,(x_n))=P$ (hence obviously $E(\I,(x_n))=S, F(\I,(x_n))=P.$). Corollary \ref{wn} implies that $\sum x_n$ is conditionally convergent.
\end{proof}

Define $A(\I,(x_n))=\left\{t \in \{0,1\}^{\N} \colon \left(\sum_{i=1}^n t(i)x_{i}\right)_n \textrm{ is } \I-\textrm{bounded} \right\}.$ In a similar way as above, thanks to Remark \ref{bw}, we get following characterisation.
\begin{corollary}
Suppose $\I$ is an ideal with the Baire property. A Banach space $X$ contains no copy of $c_0$ if and only if for each series $\sum x_n$ in $X$ which is not unconditionally convergent the set $A(\I,(x_n))$ is meager in $\{0,1\}^{\N}$.
\end{corollary}

\noindent{\bf Acknowledgement.} I would like to thank Gilles Godefroy for fruitful conversation.

\end{document}